\documentclass[11pt, twoside, a4paper]{article}
\usepackage[english]{babel}
\usepackage{amsmath,amssymb,amsthm,epsfig}
\newtheorem{theorem}{Theorem}
\newtheorem{definition}[theorem]{Definition}
\newtheorem{proposition}[theorem]{Proposition}
\newtheorem{lemma}[theorem]{Lemma}
\newtheorem{remark}[theorem]{Remark}

\def\supp{\mathop{\rm supp}\nolimits}
\def\leb{\mathop{\rm Leb}\nolimits}

\begin{document}
\title{On generic $G$-prevalent properties of $C^{r}$ diffeomorphisms of $\mathbf{S}^{1}$ and a quantitative K-S theorem}
\author{Artur O. Lopes
\footnote{Instituto de Matem\'atica, UFRGS, 91509-900 Porto Alegre, Brasil. Partially supported by CNPq, PRONEX -- Sistemas
Din\^amicos, INCT, and beneficiary of CAPES financial support.} \,\,and Elismar R. Oliveira\footnote{Instituto de Matem\'atica, UFRGS, 91509-900 Porto Alegre, Brasil. Partially supported by CNPq.}\,}
\date{\today}
\maketitle

\begin{abstract}

We will consider a convex unbounded set and certain group  of actions $G$ on this set. This will substitute the translation (by adding) structure usually consider in the classical setting of prevalence. In this way  we will be able to define the meaning of $G$-prevalent set.

In this setting we will show a kind of quantitative Kupka-Smale Theorem and also a result about rotation numbers which was first consider by J.-C. Yoccoz (and, also by M. Tsujii).

\end{abstract}


\section{Introduction} In 1992 B. Hunt, T. Sauer and J. Yorke  \cite{HSY}  introduced the idea of generic prevalent sets for a space of functions $V$ (a complete metric vector space). This concept produces a meaning for saying that a certain property is true for a  large set of functions (its complementary  being small) in the measure theoretical sense. In the next  years this idea was expanded in several different  directions and  settings \cite{KalY1}, \cite{KalY2}, \cite{KalH1}, \cite{KalH2}, \cite{OY}, \cite{SHY}, and \cite{TS1}. The main goal was to introduce the idea of \textbf{shy} sets, that are the equivalent of Christensen \cite{CRH1}, Haar zero measure sets. These will play the role of the small sets. A set will be called \textbf{prevalent} if its the complement of a \textbf{shy} set. A set (of functions, vectors, etc, that has a certain property) will be said \textbf{generic}, if it contains some \textbf{prevalent} subset.

We will consider a convex unbounded set and certain group  of action of a group $G$ on this set. This will substitute the translation (by adding) structure usually considered in the classical setting of prevalence. In this way  we will be able to define the meaning of $G$-prevalent set.

We need to elaborate a little further before explaining what we precisely mean.
A measure $\mu$ over the Borel sigma algebra which is supported in a compact set is said  to be transverse to a Borel set $S$, if $0 < \mu(U) < \infty$, for some compact set $U$, and $\mu(S + x)=0$ for all $x \in V$.
In \cite{OY}, 2005, a prevalence generic theory is characterized as satisfying 5 fundamental axioms:
\\
\textbf{Axiom 1-} A generic subset of $V$ is dense;\\
\textbf{Axiom 2-} If $P \subset Q$ and $P$ is generic then $Q$ is generic;\\
\textbf{Axiom 3-} A countable intersection of generic subset of $V$ is generic;\\
\textbf{Axiom 4-} Every translate of a generic subset is generic;\\
\textbf{Axiom 5-} A subset $P\subset \mathbb{R}^n$ is generic, if and only if, $P$ has full Lebesgue measure.\\

For prevalence in nonlinear spaces  M. Tsujii \cite{TS1} has considered translation quasi-invariant measures from $C^{r}$ sections of vector bundles in a compact $m$-dimensional smooth manifold, getting prevalent generic properties like Thom's transversality theorem. A particular setting is considered in that nice paper, but it's not clear  how to extend the construction to other families of maps and other families of measures. V. Yu. Kaloshin in \cite{KalY1}, \cite{KalY2} considered a hybrid idea of prevalence in manifolds, using topologically generic one-parametric families that has full probability on their parameter as being his approach to prevalence. This is quite interesting but it's a little bit different from the original idea of \cite{HSY}, since the prevalence should be related to a translation invariant measure. On the other hand, V. Yu. Kaloshin and B. Hunt, in \cite{KalH1}, \cite{KalH2} proved prevalent bounds for  the rate of growth of periodic orbits of diffeomorphisms in a smooth compact manifold $M$, by embedding $M$ in $\mathbb{R}^{N}$, and making perturbations of diffeomorphisms in $C^{r}(U,\mathbb{R}^{N})$, where $U$ is tubular neighborhood of $M$ in $\mathbb{R}^{N}$.

One of the main problems when considering prevalence on the space of diffeomorphisms  is that several  interesting subsets are shy, so we need an intrinsical theory of prevalence for a family of diffeomorphisms that has some additional convex structure, as, for instance, the liftings of $C^{r}$ diffeomorphisms of $\mathbf{S}^{1}$.  In this work we will introduce a different approach, replacing the group of translations (as HSY in \cite{HSY} used in the original work) by a convenient group of reflections that acts transitively on these family. We believe that our result apply directly either to the torus $\mathbb{T}^{n}$ or twist maps, without  much work.

\section{Liftings of $C^{r}$ diffeomorphisms of $\mathbf{S}^{1}$:} \label{Liftings}

Following the notation of \cite{He}, we denote $Diff^{r}_{+}(\mathbf{S}^{1})$   the set of all $C^{r}$, preserving order, diffeomorphisms of $\mathbf{S}^{1}$, for $r \geq 1$ ($Diff^{0}_{+}(\mathbf{S}^{1})$ means the set of all increasing homeomorphisms of $\mathbf{S}^{1}$). As $\mathbf{S}^{1}= \mathbb{R}/\mathbb{Z}$, we can consider the set of all liftings of order preserving diffeomorphisms, to the universal covering $\mathbb{R}$:
$$\mathcal{H}^{+}_{0}=\{F \in Diff^{r}_{+}(\mathbb{R}) \, | \, F(x+1)= F(x) + 1, \, \forall n \in \mathbb{N}, F'(x) >0 \}.$$

This set is a representative subset of the set of all $C^{r}$ maps commuting with the covering map, of degree 1:
$$\mathcal{H}^{+}=\{F \in C^{r}(\mathbb{R}) \, | \, F(x+n)= F(x) + n, \, \forall n \in \mathbb{N} \}.$$

We can also define the set of liftings of order reversing diffeomorphisms $\mathcal{H}^{-}$. As  $\mathcal{H}^{-} \cup \mathcal{H}^{+}$ is not connected by isotopy, we will restrict  ourselves, w.l.o.g., to $\mathcal{H}= \mathcal{H}^{+}$ (one can get similar results for  $\mathcal{H}^{-}$), unless we mention that. Thus, from now on, we will denote $Diff^{r}$ instead $Diff^{r}_{+}$. Also we define, $\mathcal{V}$ be the linear topological space of all $C^{r}$ function (diffeomorphisms) of $\mathbb{R}$ with the Whitney topology, so $\mathcal{H}$ is an infinite dimensional convex unbounded subset of $\mathcal{V}$. More than that, the topology induced by $\mathcal{V}$ on $\mathcal{H}$ agree with the uniform topology induced by the complete metric on $Diff^{r}(\mathbf{S}^{1})$ given by,
$$d(F,G)= \sum _{i=0}^{n}\sup_{x \in \mathbf{S}^{1}} |F^{(i)}(x) -G^{(i)}(x)|.$$

Using the natural structure of $\mathcal{V}$ as complete metric linear space we conclude from the standard prevalent theory that $\mathcal{H}$ is shy. So, for a generic $C^{r}$ function (diffeomorphism)  of $\mathbb{R}$, $F(x+n) \neq F(x) + n$, this means that $F$ is not a lifting of any diffeomorphism (or homeomorphism) of $\mathbf{S}^{1}$. Of course, this is not the right direction of reasoning. Thus, we need to introduce some intrinsical prevalent theory for $\mathcal{H}$ in order to get typical behaviors of $\mathbf{S}^{1}$ diffeomorphisms. This is the purpose of the next two sections.

We would like to observe first that $\mathcal{H}_{0}$ is not in 1-1 correspondence with the diffeomorphisms (or homeomorphisms) of $\mathbf{S}^{1}$. This happens because two liftings of the same diffeomorphism differ by an integer constant. In other words, if $\pi : \mathcal{H}_{0} \to Diff^{r}(\mathbf{S}^{1})$ is  given by $$\pi(F)=F\,(\text{mod}\, 1) =f,$$ (the canonical projection) then, $\pi(F)=\pi(G)$, implies $F=G + k, \, k \in \mathbb{Z}$. So defining the relation $F \sim G$ in $\mathcal{H}_{0}$, iff, $F=G + k, \, k \in \mathbb{Z}$, we get that $\mathcal{H}_{0}/\sim$ is a module over $\mathbb{Z}$, with the natural operations in the equivalence classes, and, the induced map $\overline{\pi}$ is an 1-1 correspondence between $\mathcal{H}_{0}/\sim$, and the set of all preserving order diffeomorphisms (or homeomorphisms) of $\mathbf{S}^{1}$.

\begin{remark} \label{PathFoliation}
Another important issue is the geometry   of $\pi (\mathcal{H}_{0})$. If we take $F, F + k, \, k \in \mathbb{Z}$ in $\mathcal{H}_{0}$, the line $\{ (1-\lambda) F + \lambda (F+k) \, | \, \lambda \in \mathbb{R}  \} $ projects on $$\pi\{ (1-\lambda) F + \lambda (F+k) \, | \, \lambda \in \mathbb{R}  \}=\{ (F +\sigma)\,(\text{mod}\, 1) \, | \, \sigma=(\lambda-[\lambda]) \in \mathbf{S}^{1}\}, $$
that is, this line projects on a closed path in $Diff^{r}(\mathbf{S}^{1})$.

On the other hand, if we take  arbitrary $F, G$ in $\mathcal{H}_{0}$, then the line $\{ (1-\lambda) F + \lambda G \, | \, \lambda \in \mathbb{R}  \} $ projects on $$\pi\{ (1-\lambda) F + \lambda G \, | \, \lambda \in \mathbb{R}  \},$$ which is, in general, a noncompact path not necessarily in $Diff^{r}(\mathbf{S}^{1})$.

For example, take $F(x)= x+ 0.2 \pi$ and $G(x)= x + 0.2 + \frac{1}{2.1\pi} \sin(2\pi x)$, we get,
$$F_{\lambda}=((1-\lambda) F + \lambda G)(x) = x + 0.2 (\pi + \lambda (1-\pi)) +\frac{\lambda}{2.1\pi} sin(2\pi x).$$
\begin{center}
\includegraphics[scale=0.3 ,angle=0]{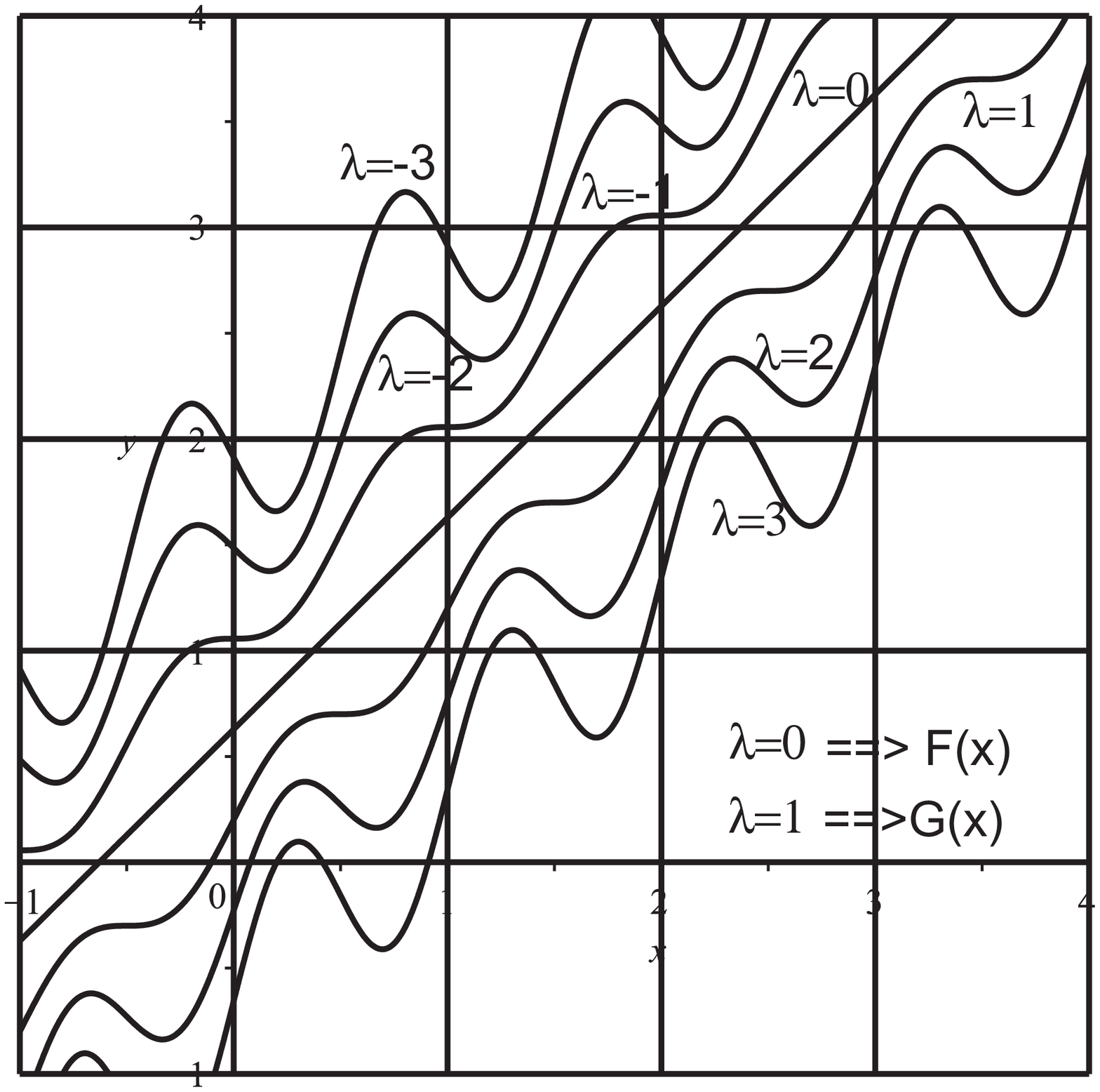}~\quad~\includegraphics[scale=0.3 ,angle=0]{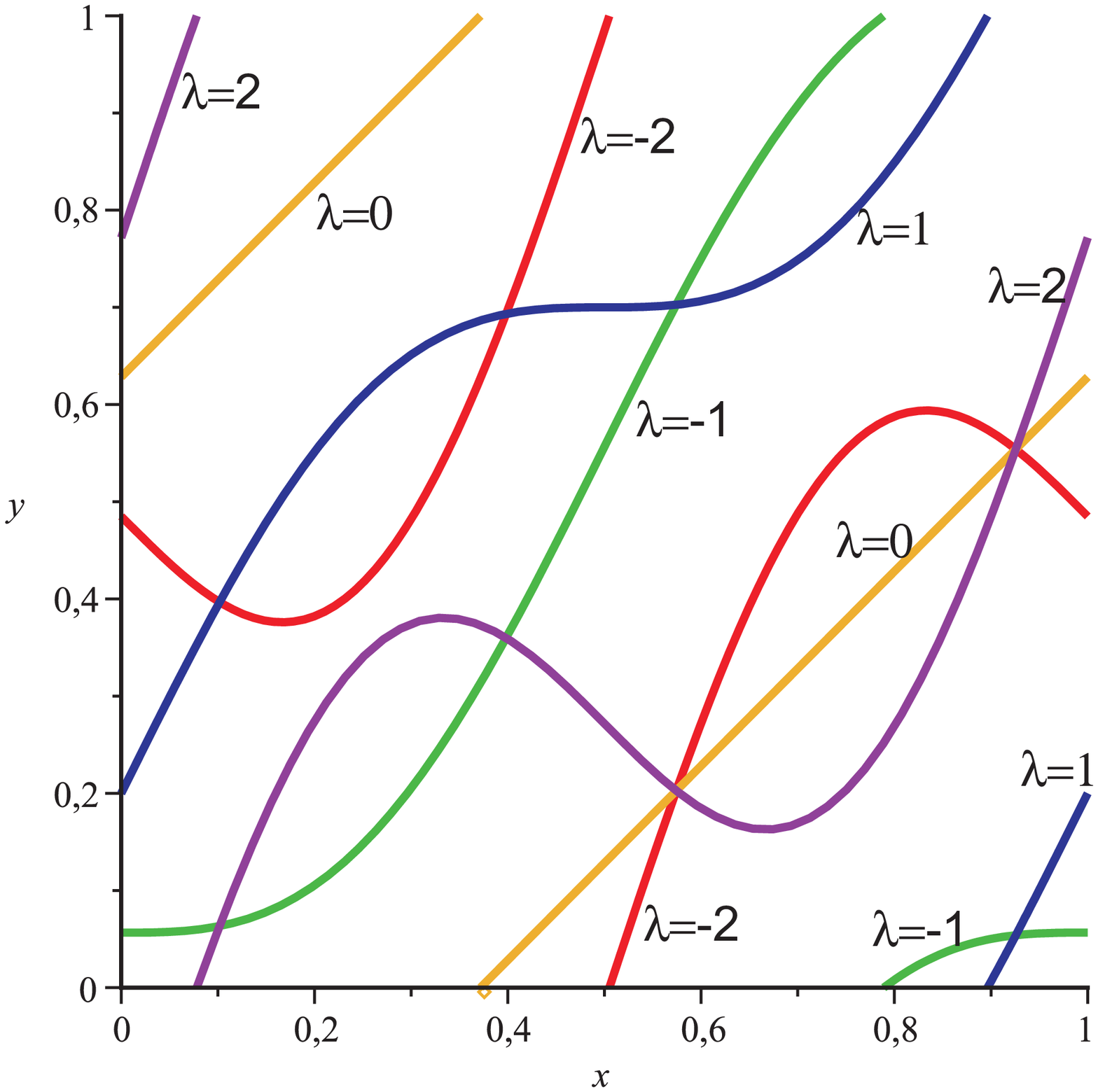}\\
\small{Path of $C^{r}$ maps  $F_{\lambda}$ ($\lambda$-foliation).}
\end{center}

We observe that if $\lambda \in [0,1]$ then all the $F_{\lambda}$  are order preserving diffeomorphisms, but when the parameter $\lambda$ grows the $F_{\lambda}$ are now part of a family of liftings of continuous maps of $S^{1}$. However, there  always exists some $\varepsilon >0$, that depends of $\min\{ \min F', \min G'\}$, such that,  $(1-\lambda) F + \lambda G \in \mathcal{H}_{0}$, for $\lambda \in (-\varepsilon,1+\varepsilon)$.
\end{remark}

\section{A group action on a convex unbounded set} \label{GroupAction}
In this section we will consider a transitive group action on a convex unbounded subset of a general metric linear space $\mathcal{V}$.

Let $(\mathcal{V}, d)$ be a complete metric linear space and $\mathcal{H} \subset \mathcal{V}$ a convex unbounded subset, that is, for any $v,w \in \mathcal{H}$  the line $\{ (1-\lambda) v + \lambda w \, | \, \lambda \in \mathbb{R}  \}$ is contained in $\mathcal{H}$.
Of course, if there exists at least three convex independent vectors $u,v,w \in  \mathcal{H}$, then it contains the 2-dimensional convex subspace $\langle u,v,w \rangle=\{au+bv+cw \, | \, a+b+c=1 \}$, what means that $\mathcal{H}$ admits Lebesgue measures supported in finite $k$-dimensional convex subspaces.

We define a set of automorphisms of $\mathcal{H}$ by:
$$\mathcal{G}=\mathcal{G}_{\mathcal{H}}=\{ \psi_{\lambda, w} : \mathcal{H} \to \mathcal{H} \, | \,  \psi_{\lambda, w}(v)=(1-\lambda) v + \lambda w,  \;  \lambda \in \mathbb{R}-\{1\} \,, w \in \mathcal{H} \}.$$
The homeomorphisms $\psi_{\lambda, w}$ are not actually ``translations" of $\mathcal{H}$. There is no sense to talk about translations in a convex space. On the other hand, the set of convex reflections  $\mathcal{G}$ provides a set of actions that acts transitively in $\mathcal{H}$, and this will be enough to develop the machinery that one needs to built a reasonable prevalence theory. Later we will define a group of transformations $G$ acting on $\mathcal{H}$.

In the future,  the property $\psi_{\lambda, w}(v)=\psi_{1-\lambda, v}(w) \in \mathcal{G}$, if $\lambda \neq 0$, will be useful.

Clearly $(\mathcal{G}, \circ)$ is not a group (because it is not closed under compositions).  For example, if there exists at least three convex independent vectors $u,v,w \in  \mathcal{H}$, then $\psi_{\frac{1}{2}, w}\psi_{-1, u}(v)= v + \frac{1}{2} w - \frac{1}{2} u$, and this cannot be written as convex combination of two vectors in $\mathcal{H}$. In the next proposition we enumerate some of the main properties of this set:
\begin{proposition}
Let $(\mathcal{G}, \circ)$ be the set defined above, then
\begin{itemize}
\item[a)] For any  $\lambda \in \mathbb{R}-\{1\}$ and  $w \in \mathcal{H}$, $\psi_{\lambda, w}$ is a homeomorphism of  $\mathcal{H}$.
\item[b)] $\psi_{0, w}=id$, for any $w \in \mathcal{H}$ (identity property).
\item[c)] For any  $\lambda,\sigma \in \mathbb{R}-\{1\}$ and  $w,u \in \mathcal{H}$, $\psi_{\lambda, w} \circ \psi_{\sigma, u} =\psi_{\sigma, u} \circ \psi_{\lambda, w}$, iff, $\psi_{\sigma, u}=id$, or, $\psi_{\lambda, w}=id$ or $u=w$ (noncommutative property).
\item[d)] For any  $\lambda \in \mathbb{R}-\{1\}$ and  $w \in \mathcal{H}$,
    $\psi_{\lambda, w} \circ \psi_{\sigma, u}  =id $, iff, $u=w$ and $ \frac{1}{\lambda} + \frac{1}{\sigma} =1$. In particular, $\psi_{\lambda, w} ^{-1} = \psi_{ \frac{\lambda}{\lambda - 1} , w}$ (inverse property). In particular also, $\psi_{\lambda, w} \circ \psi_{\sigma, u} \in \mathcal{G} $, iff, $\delta= \lambda + \sigma - \lambda\sigma \neq 0$ or $ \frac{1}{\lambda} + \frac{1}{\sigma} =1$, but with $u=w$.
\item[e)] For any  $v \in \mathcal{H}$, there exists $\lambda \in \mathbb{R}-\{1\}$ and  $\tilde{v}, w \in \mathcal{H}$, such that $\psi_{\lambda, w} \tilde{v} = v$ (transitivity).
\end{itemize}
Thus,  $\mathcal{G}_{\mathcal{H}}$ is a  noncommutative set of homeomorphisms which acts transitively on $\mathcal{H}$.
\end{proposition}
\begin{proof}$_{ }$

a) Indeed, $(1-\lambda) \tilde{v} + \lambda w = (1-\lambda) v + \lambda w$, implies $\tilde{v}=v$, because $\lambda \in \mathbb{R}-\{1\}$. So, $\psi_{\lambda, w}$ is injective. On the other hand, the equation $\psi_{\lambda, w} v= \tilde{v}$, always have a solution in $\mathcal{H}$:
\begin{align*}
(1-\lambda) v + \lambda w & = \tilde{v}\\
v & = (1- \frac{1}{(1-\lambda)}) w + \frac{1}{(1-\lambda)} \tilde{v}.
\end{align*}

b) It is trivial because  $\psi_{0, w} v = (1-0) v + 0 w = v$;\\

c) For any  $\lambda \in \mathbb{R}-\{1\}$ and  $w \in \mathcal{H}, $
\begin{align*}
\psi_{\lambda, w} \circ \psi_{\sigma, u} v & = (1-\lambda) \psi_{\sigma, u} + \lambda w \\
& = [1- \delta ]v + (\delta - \lambda) u + \lambda w ,
\end{align*}
where $\delta= \lambda + \sigma - \lambda\sigma$.

We observe that $\delta$ is a commutative combination of $\lambda$ and $\sigma$, so
$$\psi_{\sigma, u} \circ  \psi_{\lambda, w} v  = [1- \delta ]v + (\delta - \sigma) w + \sigma u.$$
Finally, $\psi_{\lambda, w} \circ \psi_{\sigma, u} =\psi_{\sigma, u} \circ \psi_{\lambda, w}$, iff, $(\delta - \lambda) u + \lambda w= (\delta - \sigma) w + \sigma u$, or, equivalently $\lambda \sigma (w-u)=0$, what means that $\psi_{\sigma, u}=id$, or $\psi_{\lambda, w}=id$, or $u=w$.\\

d) We already know that $\psi_{\lambda, w} \circ \psi_{\sigma, u} v= [1- \delta ]v + (\delta - \lambda) u + \lambda w$,
where $\delta= \lambda + \sigma - \lambda\sigma$. So, $\psi_{\lambda, w} \circ \psi_{\sigma, u} v=v, \, \forall v$, implies that, $\delta=0$ and $\lambda (w-u)=0$.
On the other hand, $\delta=0$, is equivalent to $ \frac{1}{\lambda} + \frac{1}{\sigma} =1$, and $\lambda (w-u)=0$ is equivalent to $w=u$, or $\lambda=0$, which implies $\psi_{\lambda, w}=id$.

If $\delta \neq 0$, we have, $\psi_{\lambda, w} \circ \psi_{\sigma, u} v= [1- \delta ]v + (\delta - \lambda) u + \lambda w= [1- \delta ]v +\delta [ \frac{(\delta - \lambda)}{\delta} u + \frac{\lambda}{\delta} w] \in \mathcal{G}$, because $\frac{(\delta - \lambda)}{\delta} u + \frac{\lambda}{\delta} w \in \mathcal{H}$.\\

e) Given $v \in \mathcal{H}$ we find $\lambda \in \mathbb{R}-\{1\}$ and  $\tilde{v}, w \in \mathcal{H}$ such that $\psi_{\lambda, w} \tilde{v} = v$. Therefore,
$$\psi_{\lambda, w} \tilde{v}  = v \; \Rightarrow \;
w   = \frac{1}{\lambda}v + (1 - \frac{1}{\lambda}) \tilde{v},
$$
so, for a fixed $\lambda$, we choose $w   = \frac{1}{\lambda}v + (1 - \frac{1}{\lambda}) \tilde{v}$, and, then $\psi_{\lambda, w} \tilde{v} = v$.
\end{proof}

We define now the group $(G, \circ)$ of the reflections of $\mathcal{H}$:
$$G=\{\psi_{\lambda_{1}, w_{1}} \circ \psi_{\lambda_{2}, w_{2}} \circ \dots \circ \psi_{\lambda_{n}, w_{n}} \; | \psi_{\lambda_{i}, w_{i}} \in \mathcal{G} \}.$$
Some easy computations shows that $G$ can be written as
$$G=\{\psi(v)=(1-\delta) v + \sum_{i=1}^{n}a_{i}w_{i} \; |  \delta \in \mathbb{R} -\{1\}, \; \sum_{i=1}^{n}a_{i} =\delta, \; w_{i} \in \mathcal{H}\}.$$

This group acts transitively in $\mathcal{H}$ and is topologically free:\\
\emph{Any $\psi \in G - \{Id\}$  has an unique fixed point given by the center of mass  of the vectors $w_i$'s that defines itself.}

Indeed, $\psi(v)=v$ implies that $(1-\delta) v + \sum_{i=1}^{n}a_{i}w_{i}=v$. As $\psi \neq Id$ we get $\delta \neq 0$, so $v=\frac{1}{\delta}\sum_{i=1}^{n}a_{i}w_{i}=\frac{1}{\sum_{i=1}^{n}a_{i}}\sum_{i=1}^{n}a_{i}w_{i}$. In particular, the only fixed point of $\psi_{\lambda, w}$ is $v=w$.

\section{Prevalence on $\mathcal{H}$}

The goal of this section is to define and to show the consistence of a theory which consider \emph{reflection invariant} almost every properties.

Initially, we must to define the notion of $G$- \emph{\textbf{shy}} sets in $\mathcal{H}$, as been Haar (for the group of reflections) null sets for some measure in the Borel sets  $\mathcal{H}$.

\begin{definition} \label{transverse}
Given a Borel subset $S$ we will say that a measure $\mu$ is $G$-transverse to $S$, and denote \textbf{$\mu \pitchfork S$}, if:\\
a) There exists a compact $U$, such that, $0< \mu(U) < \infty$,  ($\mu$ can be taken a compact supported measure);\\
b) $\mu( \psi S)=0$, for any $\psi \in G$.
\end{definition}

In this case we say that $S$ is $G$-shy Borel subset. The hypothesis (a) is natural from \cite{OY}, and the second one ensures that the property is  a  \emph{reflection invariant}. Analogously, a subset $S$ in $\mathcal{H}$ is said to be $G$-\textbf{shy}, if $S$ is contained in $S'$, a $G$-\textbf{shy} Borel subset  $\mathcal{H}$.

Finally, we call $G$-\textbf{prevalent} a  subset $P$ of $\mathcal{H}$, such that $P^{c}$ is $G$-\textbf{shy}.

In order to simplify the notation, from now on we will drop the letter $G$. We just point out that the concept described here is a little bit different from the usual ones.

\begin{definition}\label{GenericPrevalence}
A set $P \subset \mathcal{H}$ it said to be \textbf{generic} (in the prevalent sense) if $P$ contains some prevalent set.
\end{definition}

According to \cite{OY} the consistence of the category of the generic sets needs to satisfy the following set of axioms:\\
\textbf{Axiom 1-} A generic subset of $\mathcal{H}$ is dense;\\
\textbf{Axiom 2-} If $P \subset Q$ and $P$ is generic then $Q$ is generic;\\
\textbf{Axiom 3-} A countable intersection of generic subsets of $\mathcal{H}$ is generic;\\
\textbf{Axiom 4-} Every reflection of a generic subset is generic.\\

In order to prove Axiom 1 we need to consider the following lemma.
\begin{lemma}
Given a Borel subset $S$ and a measure $\mu$ transverse to $S$, there exists another compact supported measure $\nu$ such that:\\
a) $\nu \pitchfork S$;\\
b) The support of $\nu$ is contained in a ball of arbitrarily small radius.
\end{lemma}
\begin{proof}
Given, $\varepsilon>0$, we can take a cover of $U$, given by $\cup_{v \in U} B(v, \varepsilon)$. As  $U$ is compact, and $0< \mu(U) < \infty$, we can choose $U'=\overline{B(v_{0},\varepsilon)\cap U}$ with positive measure, and define $\nu(A)= \mu(U' \cap A)$. It is easy to see that $\nu$ satisfies the claim of the lemma.
\end{proof}

Now we are able to check Axiom 1.

\begin{theorem}
Every generic subset is dense.
\end{theorem}
\begin{proof}
If $P$ is a generic set then its complement $S$ must to be shy. We claim that  $S$ has no interior points. Indeed, if $w_{0} \in \text{int} S$, then there exists $\delta >0$  such that $B(w_{0},\delta) \subset S$. From transitivity of $G$ we can find $w   = \frac{1}{\lambda}v_{0} + (1 - \frac{1}{\lambda}) w_{0}$, then $\psi_{\lambda, w} w_{0} = v_{0}$. Using that relation $\psi_{\lambda, w}B(w_{0},\delta)=B(\psi_{\lambda, w} w_{0}, (1-\lambda )\delta)=B(v_{0}, (1-\lambda )\delta)$, we choose $\lambda \in (0,1)$, then $B(v_{0},\varepsilon) \subset B(v_{0}, (1-\lambda )\delta)$ for  $\varepsilon < (1-\lambda )\delta$. Thus, $0=\mu(\psi_{\lambda, w} S) \geq \mu(\psi_{\lambda, w} B(w_{0},\delta))= \mu(B(v_{0}, (1-\lambda )\delta)) \geq \mu(B(v_{0},\varepsilon)) >0$, and we get a contradiction.
\begin{center}
\includegraphics[scale=0.24 ,angle=0]{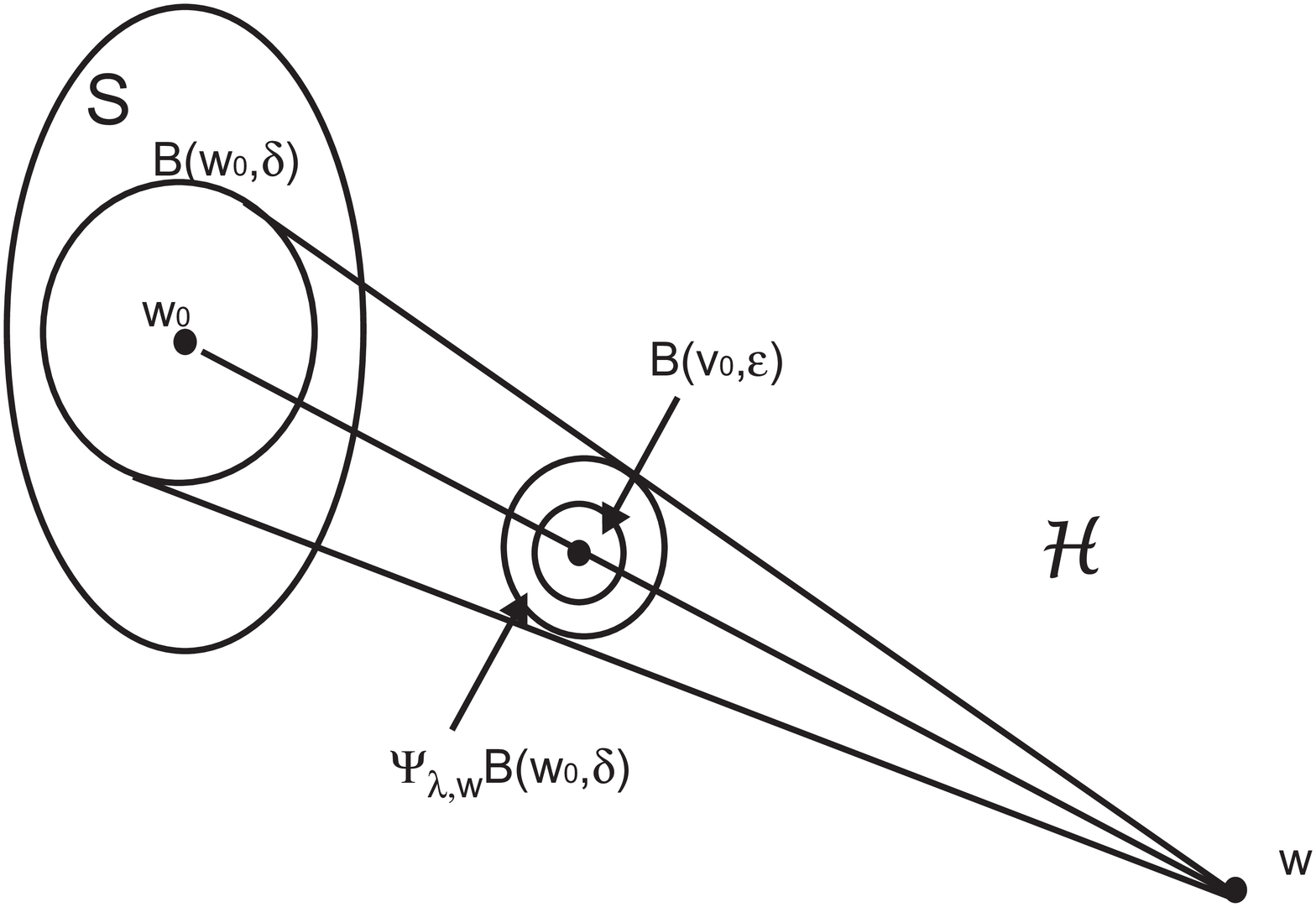}\\
\small{Reflection of $B(w_{0},\delta)$}
\end{center}

\end{proof}

The Axioms 2 and 4 are trivially obtained from the definition. Then,  we start now the proof of Axiom 3 by proving first the simplest  case:

\begin{lemma}
Let $S_{1}, S_{2}$ be shy Borel subsets and $\mu \pitchfork S_{1}$, $\nu \pitchfork S_{2}$, then there exists another compact supported measure $\mu * \nu$ such that:\\
$(\mu * \nu) \pitchfork S_{1}$, and $(\mu * \nu) \pitchfork S_{2}$.
\end{lemma}
\begin{proof}
In our proof  we adapt, in part, the idea presented in \cite{HSY}, \cite{OY}, (see \cite{HH} for convolutions on topological groups) where it is introduced the concept of a \textbf{convex convolution} of transversal measures in $\mathcal{H}$. However, in our setting  we have to face  several technical difficulties. Given $\lambda \in \mathbb{R}$, and a Borel subset $A \subseteq \mathcal{H}$, we define
$$(\mu * \nu)_{\lambda} (A)= \int_{\mathcal{H}}\int_{\mathcal{H}}\chi_{\{\psi_{\lambda,w}v \in A\}} (v,w) d\mu(v) d\nu(w).$$

In other words, for each Borel subset $A \subset \mathcal{H}$, we consider $A^{\lambda}=\{(v,w) \in \mathcal{H}^{2} \, | \, \psi_{\lambda,w}v \in A \}$, and define
$$(\mu * \nu)_{\lambda} (A)=(\mu \times \nu)(A^{\lambda}).$$

As $\psi_{\lambda, w}(v)=\psi_{1-\lambda, v}(w) \in \mathcal{G}$, if $\lambda \neq 0$ then we can rewrite the convolution as,
$$(\mu * \nu)_{\lambda} (A)= \int_{\mathcal{H}} \mu(\psi_{\frac{\lambda}{\lambda -1},w}(A)) d\nu(w),$$
or
$$(\mu * \nu)_{\lambda} (A)= \int_{\mathcal{H}} \nu(\psi_{1- \frac{1}{\lambda },v}(A)) d\mu(v).$$

The equations above ensure the second part of the definition of shyness, because:
$$(\mu * \nu)_{\lambda} \psi_{\sigma,u}(S_{1})= \int_{\mathcal{H}} \mu(\psi_{\frac{\lambda}{\lambda -1},w}\psi_{\sigma,u}(S_{1})) d\nu(w)=0,$$
and
$$(\mu * \nu)_{\lambda} \psi_{\sigma,u}(S_{2})= \int_{\mathcal{H}} \nu(\psi_{1 -\frac{1}{\lambda },v}\psi_{\sigma,u}(S_{2})) d\mu(v)=0,$$
for any $\psi_{\sigma,u} \in \mathcal{G}$.

As $\mu \pitchfork S_1$ and $\nu \pitchfork S_2$, there exists compacts sets $K_1, K_2$, such that, $0< \mu(K_{1}) < \infty$, and $0< \nu(K_{2}) < \infty$. Let $K=(1-\lambda) K_1 + \lambda K_2$. We observe that the function $\Phi: \mathcal{H}^2 \to \mathcal{H}$, given by $\Phi(v,w)=\psi_{\lambda,w}v$, is continuous, so $K=\Phi(K_1 \times K_2)$ is compact because it's the image of the compact set $K_1 \times K_2$. Now we compute $(\mu * \nu)_{\lambda}(K)$:
$$(\mu * \nu)_{\lambda}(K)=(\mu \times \nu)(K^{\lambda})=(\mu \times \nu)(K_1 \times K_2)= \mu(K_{1}) \nu(K_{2}).$$
\end{proof}

The extension of this result for the case of a general finite union of shy sets is not a direct consequence of the previous one because $G$ is noncommutative. We have to proceed in a combinatorial way:\\

First, we observe that $\psi_{\lambda, w}(v)=\psi_{1-\lambda, v}(w) \in \mathcal{G}$ implies the remarkable property:
$$\{ \psi_{\lambda_{n-1}, w_{n}} \psi_{\lambda_{n-2}, w_{n-1}} ... \psi_{\lambda_{k-1}, w_{k}} ... \psi_{\lambda_{1}, w_{2}}(w_{1}) \in A \}= \{ w_{k} \in \psi A,\; \psi \in G \}.$$

In order to see this first we take, $$\varphi=\psi_{\lambda_{n-1}, w_{n}}... \psi_{\lambda_{k-2}, w_{k-1}}, \, \zeta =\psi_{\lambda_{k }, w_{k+1}} ... \psi_{\lambda_{1}, w_{2}} \in G,$$
\begin{align*}
\psi_{\lambda_{n-1}, w_{n}} \psi_{\lambda_{n-2}, w_{n-1}} ... \psi_{\lambda_{k-1}, w_{k}} ... \psi_{\lambda_{1}, w_{2}}(w_{1}) & = \varphi \psi_{\lambda_{k-1}, w_{k}} \zeta (w_{1})= z  \\
\psi_{\lambda_{k-1}, w_{k}} \zeta (w_{1}) & = \varphi^{-1} z  \\
\psi_{1-\lambda_{k-1}, \zeta (w_{1}) } (w_{k}) & = \varphi^{-1} z\\
w_{k} & =\psi_{1-\lambda_{k-1}, \zeta (w_{1}) }^{-1}\psi_{\varphi}^{-1} z,
\end{align*}
then we take $\psi=\psi_{1-\lambda_{k-1}, \zeta }^{-1}\psi_{\varphi}^{-1} \in G $.\\

This argument ensures that, if $\mu_{i} \pitchfork S_i$, $i=1,2, ...,n$, then $$(\mu_{n} * \mu_{n-1} * ...\mu_{1})(\psi S_i)=0, \, \forall \psi \in G.$$

To prove that $(\mu_{n} * \mu_{n-1} * ...\mu_{1})_{\lambda} (K)>0$, for some compact $K$, the same argument will work, we just point out that the function $\Phi: \mathcal{H}^2 \to \mathcal{H}$ given by $\Phi(w_1,w_2, ...,w_n)=\psi_{\lambda_{n-1},w_{n}} \circ ..\circ \psi_{\lambda_1,w_2}w_1$ is continuous, so $K=\Phi(K_1 \times ... \times K_n)$ is compact as the image of the compact $K_1 \times ... \times  K_n$, and if we compute $(\mu_{n} * \mu_{n-1} * ...\mu_{1})_{\lambda}(K)$ we get
$$(\mu_{n} * \mu_{n-1} * ...\mu_{1})_{\lambda}(K)= \mu_{1}(K_{1}) ...\mu_{n}(K_{n}).$$ In order to prove that Axiom 3 is true, we need to understand how to control the process of countable convex convolution, and this will be done next.

\begin{theorem}\label{countableSHY}
If $\mu_{i} \pitchfork S_{i}, \; i=\{1,2,...\}$, then there exists a measure $\mu$ such that $\mu \pitchfork S_{i}, \; i=\{1,2,...\}$.
\end{theorem}

\begin{proof} In order to prove this, we introduce the \textbf{Countable Convex Convolution (CCC)} of $\{ \mu_n\}$ (see \cite{HH} for tight sequences and infinite convolutions):\\

Consider $\mathcal{H}^{\pi}=\mathcal{H} \times \mathcal{H} \times \mathcal{H} \times ....$, and let $K=K_{1} \times K_{2} \times ...$, where $K_{n}$ is the support of each $\mu_n$, by Tychonov theorem we know that $K$ is compact; additionally, we suppose that each $\mu_n$ is a measure of probability. We choose a sequence of $\psi_{\lambda_{n}, w} \in \mathcal{G}$ and define, for each measurable $A \subset \mathcal{H}$ the set:
$$A^{\lambda}=\{ (w_1, w_2,...)\in \mathcal{H}^{\pi} \, | \, ... \circ\psi_{\lambda_{n-1}, w_n} \circ ... \circ\psi_{\lambda_{1}, w_2}(w_1) \in A\},$$
thus, the CCC probability of $A$ is
$$(... * \mu_{n} * ... * \mu_{1})_{\lambda} (A)=(\mu_{1} \times \mu_{2} \times ...)(A^{\lambda}).$$

On the other hand,  $A^{\lambda}=\Phi^{-1}(A)$, where $\Phi:\mathcal{H}^{\pi} \to  \mathcal{H}$ is given by:
$$\Phi(w_1, w_2,...)= ... \circ\psi_{\lambda_{n-1}, w_n} \circ ... \circ\psi_{\lambda_{1}, w_2}(w_1).$$

Now, we must show that $\Phi$ is well defined and continuous. Consider $\lambda_{n}=\frac{1}{\pi^2 n^2}, \, n \in \mathbb{N},$
then
\begin{align*}
\Phi(w_1, w_2,...)&=... \circ\psi_{\lambda_{n-1}, w_n} \circ ... \circ\psi_{\lambda_{1}, w_2}(w_1)\\
&= \lim_{n \to \infty} p_{n-1} w_{1} + p_{n-1}\frac{\lambda_{1}}{p_{1}} w_{2}+ p_{n-1}\frac{\lambda_{2}}{p_{2}} w_{3}+ ... + p_{n-1}\frac{\lambda_{-1}}{p_{n-1}} w_{n}\\
&=\lim_{n \to \infty} p_{n-1} \left[ w_{1} +  \sum_{i=1}^{n-1} \frac{\lambda_{i}}{p_{i}} w_{i+1}\right],\\
\end{align*}
where $p_{n}=\prod_{i=1}^{n} (1-\lambda_{i})$. From Euler's theorem for infinite products, we know that
$$\sin(z)=z \prod_{n=1}^{\infty} (1-\frac{z^2}{\pi^2 n^2}), \forall z \in \mathbb{C},$$
and, taking $z=1$, we get $\lim_{n \to \infty} p_{n-1}=\sin(1)$, what means
$$\Phi(w_1, w_2,...)= \sin(1) w_1 + \frac{\sin(1)}{\pi^2 (1-\frac{1}{\pi^2})} w_2 + \frac{\sin(1)}{4 \pi^2 (1-\frac{1}{\pi^2})(1-\frac{1}{4 \pi^2})} w_3 + ...$$
 $\Phi$ is, obviously, a continuous map in the product topology of $\mathcal{H}^{\pi}$.

Finally, we observe that $$\Phi(K) = \sin(1) K_1 + \frac{\sin(1)}{\pi^2 (1-\frac{1}{\pi^2})} K_2 + \frac{\sin(1)}{4 \pi^2 (1-\frac{1}{\pi^2})(1-\frac{1}{4 \pi^2})} K_3 + ...$$ is a compact set in $\mathcal{H}^{\pi}$,  and
$(... * \mu_{n} * ... * \mu_{1})_{\lambda} (K)=\mu_{1}(K_1)\mu_{2}(K_2) ... = 1$. It's also easy, to see that for a fixed $i$,  $(\mu_{n} * \mu_{n-1} * ...\mu_{1})(\psi S_i)=0, \, \forall \psi \in G, \forall n\geq i$, so $(... * \mu_{n} * ... * \mu_{1})_{\lambda}(\psi S_i)=0, \, \forall \psi \in G.$ Thus $(... * \mu_{n} * ... * \mu_{1})_{\lambda} \pitchfork S_i$, $i=\{1,2, ...\}$.\\

\end{proof}

\vskip 2mm

\section{Conditional prevalence}

We use the term conditional prevalence instead relative prevalence because this term is already used in the literature for prevalence trough translations by a prevalent set.
The set $\mathcal{H}_{0}$ of all liftings of circle diffeomorphisms is defined by an open condition ($F' >0$) in $\mathcal{H}$, the set of all liftings of degree 1 of $\mathbf{S}^{1}$, but it is clear that any subset $S \subset \mathcal{H}_{0}$ could not to be shy, since $\mathcal{H}_{0}$ is not invariant under the action of $G$ (see for example, Remark~\ref{PathFoliation}).

In a future work we will analyze properties of maps on the circle, and potentials, so the general theory that we develop here will be useful in this task. However our focus is to study properties of diffeomorphisms   so we introduce  the idea of \textbf{Conditional prevalence}.

Given a set $\mathcal{H}_{0} \subset \mathcal{H}$, we denote
$$G_{0}=\{\psi_{\lambda_{1}, w_{1}} \circ \psi_{\lambda_{2}, w_{2}} \circ \dots \circ \psi_{\lambda_{n}, w_{n}} \; | w_{i} \in \mathcal{H}_{0}\} \supset \mathcal{G}_{0}.$$

\begin{definition} \label{Conditional tranverse}
Given a Borel subset $S \subset \mathcal{H}_{0}$ we will say that a measure $\mu$ is condional $G$-transverse to $S$ in $\subset \mathcal{H}_{0}$, and denote \textbf{$\mu \pitchfork_{\mathcal{H}_{0}} S$} if:\\
a) There exists a compact $U \subset \mathcal{H}_{0}$, such that, $0< \mu(U) < \infty$,  ($\mu$ can be taken a compact supported measure in $\mathcal{H}_{0}$);\\
b) $\mu( \psi S)=0$, for any $\psi \in G_{0}$.
\end{definition}

In this case we say that $S$ is a conditional $G$-shy Borel subset, and, we call conditional $G$-\textbf{prevalent} a  subset $P$ of $\mathcal{H}_{0}$, such that $P^{c}$ is $G$-\textbf{shy}. In order to avoid complicated notations we will just say that $S$ is shy in $\mathcal{H}_{0}$, or $S$ is prevalent in $\mathcal{H}_{0}$.

It is easy to see that all the above theory works well in this setting, even the CCC and, the density of conditional prevalent sets, indeed in the proof of Axiom 3, we get, by contradiction, $B(w_{0},\delta) \subset S \subset \mathcal{H}_{0}$ and, $\psi_{\lambda, w}B(w_{0},\delta) \supset \mu(B(v_{0},\varepsilon)) \supset \supp \mu$, what shows that $\supp \mu$ is stable under $\psi_{\lambda, w}^{-1}$, then the conclusion $\mu(\psi_{\lambda, w} B(w_{0},\delta))=0$ is  still valid. We observe that the conditional convex convolution $(\mu * \nu)_{\lambda}$ is well defined since the parameter $\lambda$ can be chosen  in such way that  $\supp (\mu * \nu)_{\lambda}=K=\Phi(K_1 \times K_2) \subset \mathcal{H}_{0}$ for  $\Phi: \mathcal{H}^2 \to \mathcal{H}$ given by $\Phi(v,w)=\psi_{\lambda,w}v$ (taking $\lambda \in (0,1)$, for example), the same is true for CCC.

Another important fact, is that Axiom 1 implies that, if $\pi : \mathcal{H}_{0} \to Diff^{r}(\mathbf{S}^{1})$  given by $\pi(F)=F (\text{mod} \,1)$ is the canonical projection, then $\pi(P)$ is dense in $Diff^{r}(\mathbf{S}^{1})$,  for any prevalent $P \subset \mathcal{H}_{0}$.

\begin{definition}
A set $ \mathcal{O} \subset Diff^{r}(\mathbf{S}^{1})$ is called \textbf{prevalent}, if $\pi^{-1}(\mathcal{O}) \subset \mathcal{H}_{0}$ is a prevalent subset of $\mathcal{H}_{0}$.
\end{definition}

\section{Transversal invariant measures}

The main goal in this section is to get probabilities on $\mathcal{H}$ that are ``invariant" under the group action $G$.

The hard part to show that a certain given set $S$ is transversal to a Borel measure  $\mu$ is to show that $\mu( \psi S)=0$, for any $\psi \in G$. So we can formulate the following question: ``How one can obtain examples of $G$-invariant measures, that is, a Borel measure $\mu$ such that, for each Borel subset $S \subset \mathcal{H}$, with $\mu(S)=0$, we have $\mu( \psi_{\lambda, w} S)=0$, for any $\psi_{\lambda, w} \in \mathcal{G}$?"

\begin{theorem}\label{Invariant Measure}
Let $V=C^{0}([0,1])$ the topological vector space where we consider the uniform topology. Then, for each $n \in \mathbb{N}$, there exists a Borel measure with compact support $\mu_{n}$ in $V$, that is, $G^{1}$-invariant measure ($G$ is the group generated by $\mathcal{G}$ of reflections $\psi_{\lambda, w}$ with $\lambda \neq 1$). In particular $\mu_{n} \pitchfork S$ for every zero measure set $S$.
\end{theorem}
\begin{proof}
Define for each $x \in [0,1]$ the linear functional $x: C^{0} \to [0,1]$ given by
$$x(f)=f(x),$$
and consider the partition of $[0,1]$ given by $\{\frac{0}{n}, \frac{1}{n}, ...\frac{n}{n}\}$. Then, we can obtain the open map $\rho : C^{0} \to \mathbb{R}^{n+1}$:
$$\rho(f)=(\frac{0}{n}(f), \frac{1}{n}(f), ...\frac{n}{n}(f)).$$
Now, we consider $\leb$ the Lebesgue measure on $[0,1]$, that is, $\leb(A) =\int_{\mathbb{R}} \chi_{A}(x) dx$, and the product measure $\mu_{0}=\leb^{n+1}$ on $[0,1]^{n+1}$.  Thus we are able to define a Borel measure $\mu_{n}$ in $C^{0}$ by the push forward
\begin{align*}
\mu_{n}(B)&=\mu_{0} (\rho (B)) \\
&= \leb(\frac{0}{n}(B)) \cdot ... \cdot \leb(\frac{n-1}{n}(B)) \cdot\leb(\frac{n}{n}(B)) \\
&= \int_{\mathbb{R}} \chi_{\frac{0}{n}(B)}(x)dx \cdot ...\cdot \int_{\mathbb{R}}\chi_{\frac{n-1}{n}(B)}(x)dx \cdot \int_{\mathbb{R}}\chi_{\frac{n}{n}(B)}(x)dx.
\end{align*}
\begin{center}
\includegraphics[scale=0.20,angle=0]{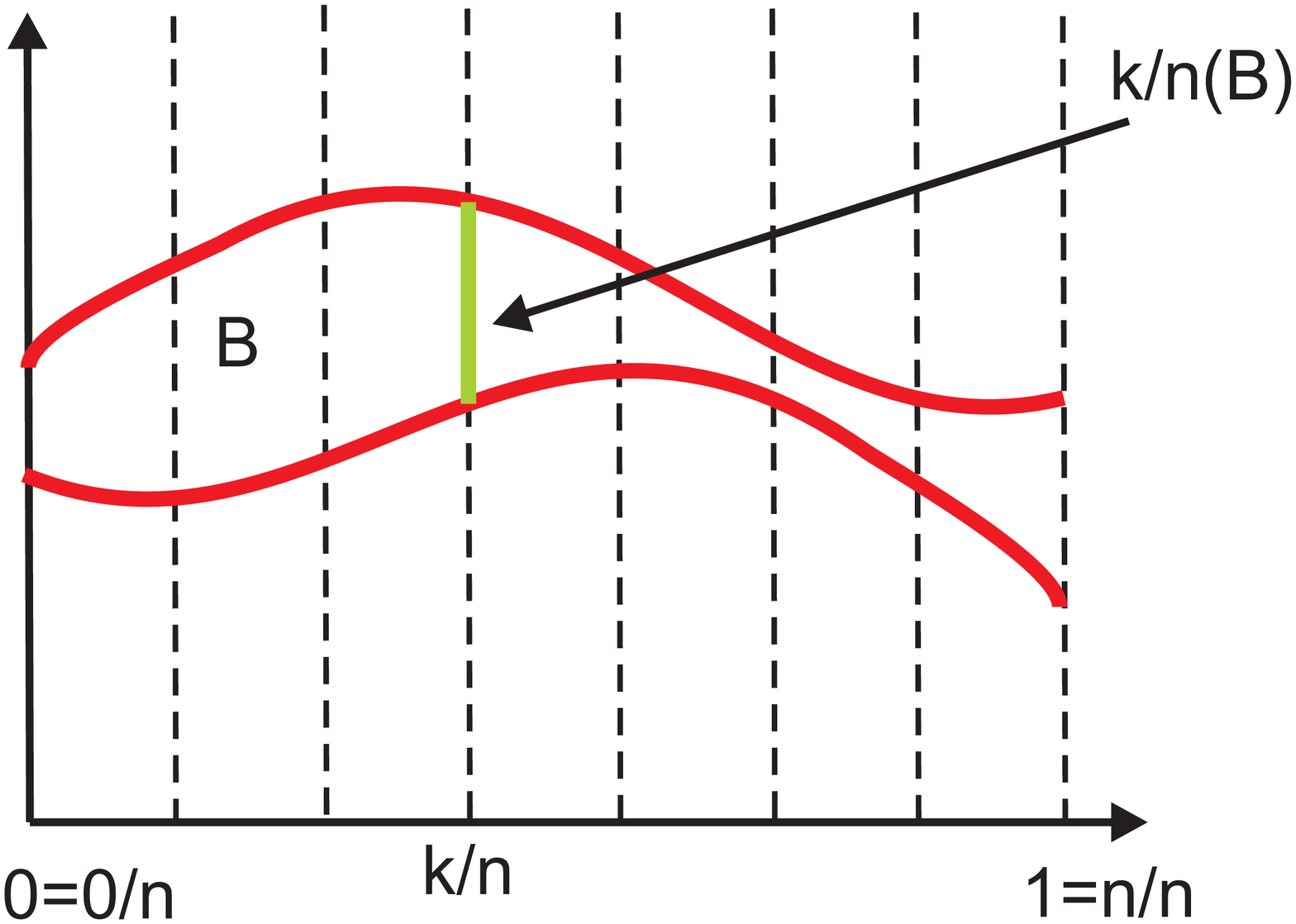}\\
\small{Push forward of the product measure.}\\
\end{center}
 Obviously, $\mu_{n}$ has compact support, we claim that $\mu_n$ is $G$-invariant. Indeed, suppose $\mu_{n}(A)=0$, so there exists $k$ such that, $\leb(\frac{k}{n}(A))=0$, and for any $\psi_{\lambda, h} \in \mathcal{G}$ we have
\begin{align*}
\leb(\frac{k}{n}(\psi_{\lambda, h}A))&=\leb(\frac{k}{n}((1-\lambda) A + \lambda h)) \\
&=\leb((1-\lambda) \frac{k}{n}(A) + \lambda \frac{k}{n}(h))) \\
&=\int_{\mathbb{R}} \chi_{(1-\lambda) \frac{k}{n}(A) + \lambda h(k/n)}(x)dx\\
&=\int_{\mathbb{R}} \chi_{(1-\lambda) \frac{k}{n}(A)}(y)dy, \; y=x -\lambda h(k/n)\\
&=(1-\lambda) \int_{\mathbb{R}} \chi_{\frac{k}{n}(A)}(z)dz, \; (1-\lambda)z=y\\
&=(1-\lambda)\leb(\frac{k}{n}(A))=0,
\end{align*}
thus $\mu_{n}(\psi_{\lambda, h} A)=0$.
\end{proof}


\section{Evaluation maps for periodic points} \label{evaluationSection}
In this section, and in the next one, we will consider $$\mathcal{H}_{0}=\text{Liftings}(Diff^{r}_{+}
(\mathbf{S}^{1})),$$ for $r \geq 2$, so the prevalence results are on this set. We need regularity $r \geq 2$ in order to apply evaluations technics. The main goal of this section is to provide the technical tools that we need for applications, more specifically we will construct and analyze properties of one dimensional probes, that is, finitely dimensional spaces that supports transversal measures.

For each fix $F,G \in \mathcal{H}_{0}$ we consider the path $\alpha_{\lambda}:\mathbb{R} \to \mathcal{H}$ given by $$\alpha_{\lambda}(x)=(1-\lambda)F(x)+ \lambda G(x).$$

We define the \textbf{$n$-evaluation map}, associated to $\alpha_{\lambda}$, as the multivalued function:
$$\Delta_{n}(x)=\{ \lambda \, | \, \alpha_{\lambda}^{n}(x)=x\text{ and } x \in S^{1} \},$$
which can be evetually empty.

We observe that the lifting property implies that $\Delta_{n}(x)$ can be extended as a time 1, periodic function in $\mathbb{R}$. Moreover, if $\alpha_{\lambda} \in \mathcal{H}_{0}$ then $\Delta_{n}(x)=\{\lambda\}$, what means that $\lambda=\Delta_{n}(x)$ is locally a $C^{r}$ function. To see this, we define the $C^{r}$ map $\varphi(x,\lambda)= \alpha_{\lambda}(x)$, and denote $\varphi^{n}(x,\lambda)=\varphi(\varphi(...,\lambda),\lambda)$, so $\alpha_{\lambda}^{n}(x)=\varphi^{n}(x,\lambda)$.
Differentiating with respect to $\lambda$ we get
$$\frac{\partial \varphi^{n}}{\partial \lambda}(x,\lambda)= \frac{1}{\frac{\partial \varphi }{\partial x}} \sum_{k=0}^{n-1} (G-F)(\varphi^{k})\frac{\partial \varphi^{n-k}}{\partial x} \neq 0,$$
because $\alpha_{\lambda} \in \mathcal{H}_{0}$, and $G-F \neq 0$.

To establish the domain of $\lambda=\Delta_{n}(x)$ we need to assume that $F$ has some fixed point, then we observe that $\alpha_{\lambda} \in \mathcal{H}_{0}$ for some interval $\lambda \in (a,b) \supset [0,1]$ (for example, for $F(x)=x + 0.1 + \frac{1}{2.1\pi} \sin(2\pi x)$ and $G(x)= x+ 0.2 \pi $, we get, $(a,b)=(-0.05,2.05)$ as the maximal interval for this property), then we define $D^{n}(F,G)=\Delta_{n}^{-1} (\Delta_{n} (S^{1}) \cap (a,b))$, because $\Delta_{n} (S^{1}) \cap (a,b)=\{\lambda \, | \, \alpha_{\lambda}(F) \in \mathcal{H}_{0},\text{ and there exists } x \, \text{ s.t. } \alpha_{\lambda}(F)^{n}(x)=x \}$.
Thus $\lambda=\Delta_{n}:D^{n}(F,G)\to (a,b)$ represents all the periodic points of period $n$, for $\alpha_{\lambda}$ in the following sense:

\emph{``$p \in S^{1}$ is a periodic point of period $n$ for the diffeomorphism $\alpha_{\lambda}$ with degenerate derivative $(\alpha_{\lambda}^{n})'(p)=1$ if, and only if, $\lambda=\Delta_{n}(p) \in (a,b)$ and $\Delta_{n}'(p)=0$"}.

We like to observe that some particular cases will be useful in the next sections:\\

\textbf{Type I -}  $F \in \mathcal{H}_{0}$ with fixed points, and $G=F+k$ $k \in (0,1)$.
In this case, $\alpha_{\lambda}(x)=(1-\lambda)F(x)+ \lambda G(x)= F(x) + \lambda k \in \mathcal{H}_{0}, \, \forall \lambda \in \mathbb{R}$. Thus, the evaluation $\lambda=\Delta_{n}$ is fully defined and the study of the hyperbolic periodic points reduces to the study of the critical points of the evaluation. More than that, applying an element $\psi \in G_{0}$, we get
$$ \psi(\alpha_{\lambda})(x)=\alpha_{(1-\delta)\lambda}(x)= \psi(F)(x)+ (1-\delta)\lambda k,$$
where, $\psi(F)=(1-\delta) F + \sum_{i=1}^{n}a_{i}G_{i} \; |  \delta \in \mathbb{R} -\{1\}, \; \sum_{i=1}^{n}a_{i} =\delta, \; G_{i} \in \mathcal{H}_{0}$. Thus, there are just two possibilities,  $\psi(\alpha_{\lambda}) \subset  \mathcal{H}_{0}$ if  $\psi(F) \in  \mathcal{H}_{0}$,  or $\psi(\alpha_{\lambda}) \subset \mathcal{H}- \mathcal{H}_{0}$ if  $\psi(F) \not\in  \mathcal{H}_{0}$.\\

\textbf{Type II -} $F,G \in \mathcal{H}_{0}$, $F$ with fixed points, and $\displaystyle \min_{x \in S^{1}}|G(x)-F(x)|=\sigma >0 $. In this case, $\alpha_{\lambda}$ is a $\lambda$-foliation of $\mathbb{R}^{2}$ (see Remark\ref{PathFoliation}, for example). Let $(a,b)$ be the maximal interval such that $\alpha_{\lambda} \in\mathcal{H}_{0}$ and $D^{n}(F,G)$ the domain of the evaluation map $\lambda=\Delta_{n}$, we need to compute the derivative $\frac{d \lambda}{dx}$. Using the notation $\varphi(x,\lambda)= \alpha_{\lambda}(x)$, and $\varphi^{n}(x,\lambda)=\varphi(\varphi(...,\lambda),\lambda)$, where $\lambda=\lambda(x)$, and $\alpha_{\lambda}^{n}(x)=x$, we get
$$1= \frac{\partial \alpha_{\lambda}^{n}}{\partial x}(x,\lambda)=\frac{\partial \varphi^{n}}{\partial x}(x,\lambda)= \frac{d\varphi^{n}}{dx} + \lambda'(x) \frac{d\varphi^{n}}{dx} \sum_{k=0}^{n-1} \frac{(G-F)(\varphi^{k})}{\frac{d\varphi^{k+1}}{dx}}.$$
Isolating $\lambda'$, and using $\varphi(x,\lambda)= \alpha_{\lambda}(x)$, we get
$$ \lambda'(x) = \frac{1- \frac{d\alpha_{\lambda}^{n}}{dx}}{\frac{d\alpha_{\lambda}^{n}}{dx} \sum_{k=0}^{n-1} (\frac{d\alpha_{\lambda}^{k+1}}{dx})^{-1}(G-F)(\alpha_{\lambda}^{k})}.$$

The main application of the evaluation map is to describe the behavior of the hyperbolic periodic points. Indeed, the formula above shows that for a diffeomorphism $\alpha_{\lambda}$, one has $\frac{d\alpha_{\lambda}^{n}}{dx}(x)=1$ only if $\lambda'(x)=0$ and $\lambda=\Delta_{n}(x)$.

For type II probes the action of $G$ is not trivial and the interval $(c,d)$ where $\psi(\alpha_{\lambda}) \in\mathcal{H}_{0}$ can be different of $(a,b)$. Anyway, the same conclusions are true in a possibly different domain. So the conditions $\psi(\alpha_{\lambda})^{n}(p)=p  \text{ and } \frac{d}{dx}\psi(\alpha_{\lambda})^{n}(p)=1,$  are equivalent to $\lambda=\Delta_{n}(p)  \text{ and } \Delta_{n}'(p)=0,$
because,
$$\frac{\partial}{\partial x} \psi(\alpha_{\lambda(x)})^{n}(x) + \frac{\partial}{\partial \lambda} \psi(\alpha_{\lambda(x)})^{n}(x) \lambda'(x)= 1,$$
thus,
$$
\lambda'(p)=\Delta_{n}'(p)=\frac{1-\frac{\partial}{\partial x} \psi(\alpha_{\lambda(p)})^{n}(p)}{
\frac{\partial}{\partial \lambda} \psi(\alpha_{\lambda(p)})^{n}(p)}.
$$

\section{Applications of prevalence on $\mathbf{S}^{1}$}

The first result we will get is about the non-degeneracy of the fixed points of a  diffeomorphism of $\mathbf{S}^{1}$. This result is contained in the Kupka-Smale theorem, but we want to get this result in our setting. We present a direct proof here which is  instructive, and, at the same time, shows a different choice of probe for getting the desired property.

We say that a fixed point $p$ is  non-degenerate (of order 1) if $F'(p) \neq 1$.

\begin{proposition}\label{ElemntaryDiff} The order preserving diffeomorphisms of $\mathbf{S}^{1}$ are generically nondegenerate, in particular they are hyperbolic.
\end{proposition}
\begin{proof}
Let $S=\{F \in \mathcal{H}_{0}  \, | \, \exists p \in Fix(F), \; \frac{d}{dx}F(p)=1 \}$, and we must show that $S$ is shy, that is, the set of all diffeomorphisms without degenerate fixed point is prevalent.

The proof will be accomplish  by choosing a special one dimensional probe of type II, in $\mathcal{H}_{0}$. We consider the one dimensional subspace generated by $F$ and $H$,
$$\alpha(x,\lambda)=\psi_{\lambda,H}F=(1-\lambda)F + \lambda H, \; \lambda \in [0,1],$$
and consider $\mu$ the Lebesgue measure supported on the trace of $\alpha$, where $\mu(A)$ means $\mu(\{\lambda \, | \, \alpha_{\lambda} \in A\})$. We claim that $\mu(S \cap \alpha_{\lambda})=0$. We will not explain now this computations because it is a particular case of the next one ($\psi=Id$).

In order verify the condition of invariance we need to show that $\mu(\psi(S) \cap \alpha)=0$,  or equivalently, $\mu(S \cap \psi(\alpha))=0$, for all $\psi \in G$.
Remember that  $\psi(F)=(1-\delta) F + \sum_{i=1}^{n}a_{i}G_{i} \; |  \delta \in \mathbb{R} -\{1\}, \; \sum_{i=1}^{n}a_{i} =\delta, \; G_{i} \in \mathcal{H}_{0}$, so we have
$ \displaystyle \psi(\alpha)(x,\lambda)= (1-\delta) F(x) +  \lambda(1-\delta) (H(x) - F(x)) + \sum_{i=1}^{n}a_{i}G_{i}(x).$
Now, $\psi(\alpha)(p,\lambda) =p$ implies that, $1=(1-\delta) F'(p) +  \lambda(1-\delta) (H'(p) - F'(p)) + \sum_{i=1}^{n}a_{i}G'_{i}(p).$
Thus,
$$\Delta(x)=\frac{x-F(x)}{H(x) -F(x)} + \frac{1}{H(x) -F(x)} \frac{\delta}{1-\delta} \left(\frac{1}{\delta} \sum_{i=1}^{n}a_{i}G_{i}(x)  - x \right).$$

Substituting the two formulas above in $\Delta'(x)$ we conclude  that $\Delta'(p)=0$, and, proceeding as before, we get
\begin{align*}
\{\lambda \, | \, S \cap \psi(\alpha_{\lambda})\} & = \{\lambda \in [0,1] \, | \, \exists p \in Fix(\psi(\alpha_{\lambda})(\cdot, \lambda)), \; \frac{d}{dx}\psi(\alpha_{\lambda})(p, \lambda)=1 \}\\
& \subseteq \Delta(\{ p \, | \, \Delta'(p)=0 \}),\\
\end{align*}
which has zero Lebesgue measure by Sard's theorem.
So $\mu \pitchfork S$, that is $S$ is shy.
\end{proof}

Applying thess technics we can prove a kind of Kupka-Smale theorem for  non-degeneracy of the periodic points of a preserving order diffeomorphism of $\mathbf{S}^{1}$.

We remember that a periodic point $p$ of order $n$ ($F^{n}(p)=p$) is said to be hyperbolic, if $|\frac{d}{dx}F^{n}(p)| \neq 1$; for a order preserving transformation  this is equivalent to say that every iterate $F^{n}$ has no degenerate fixed points.  A diffeomorphism $F$ is said K-S (Kupka-Smale type) if all periodic points are hyperbolic.

\begin{theorem}\label{KS OrientPreservDiff} The order preserving diffeomorphisms of $\mathbf{S}^{1}$ are generically K-S.
\end{theorem}
\begin{proof}
Let $S_{n}=\{F \in \mathcal{H}_{0}  \, | \, \exists p \in Fix^{n}(F), \; \frac{d}{dx}F^{n}(p)=1 \}$ be the set of diffeomorphisms such that the  iterate $F^{n}$ has at least one degenerate fixed point. We must to show that $S_{n}$ is shy, so $$S= \bigcup_{n=1}^{\infty} S_{n}= \{F \in \mathcal{H}_{0}  \, | \, \text{has some degenerate periodic point.}\},$$ will be shy, that is, the set of all  diffeomorphisms  which preserve order and also without degenerate periodic  point $\mathcal{H}_{0}- S$, is prevalent.

The proof will follow easily by  choosing a one dimensional probe of type I,
$$\alpha_{\lambda}(.)=\psi_{\lambda,H}F=(1-\lambda)F + \lambda H= F + \lambda k , \; \lambda \in [0,1],$$
and denote $\mu$ the Lebesgue measure supported on the trace of $\alpha$, where $\mu(A)$ means $\mu(\{\lambda \, | \, \alpha_{\lambda} \in A\})$.

We claim that $\mu(\psi(S_{n}) \cap \alpha_{\lambda})=0$,  or, equivalent, $$\mu(\{\lambda|S_{n} \cap \psi(\alpha_{\lambda})\neq \emptyset\})=0,$$ for all $\psi \in G$.
In order to see this, remember that from Section \ref{evaluationSection} the conditions $$\psi(\alpha_{\lambda})^{n}(p)=p  \text{ and } \frac{d}{dx}\psi(\alpha_{\lambda})^{n}(p)=1,$$
are equivalent to, $\lambda=\Delta_{n}(p)  \text{ and } \Delta_{n}'(p)=0,$ thus,
$$
\lambda'(p)=\Delta_{n}'(p)=\frac{1-\frac{\partial}{\partial x} \psi(\alpha_{\lambda(p)})^{n}(p)}{
\frac{\partial}{\partial \lambda} \psi(\alpha_{\lambda(p)})^{n}(p)}.
$$

We have two possibilities, if $\psi(F) \not\in \mathcal{H}_{0}$   then also $\psi(\alpha_{\lambda}) \not\in \mathcal{H}_{0}$, so  $\psi(\alpha_{\lambda }) \cap S_{n}$ has zero measure. Otherwise, if  $\psi(F) \in \mathcal{H}_{0}$,   then also $\psi(\alpha_{\lambda}) \in \mathcal{H}_{0}$.

Proceeding as before, we get
$$
\{\lambda \, | \, S_{n} \cap \psi(\alpha_{\lambda}) \}= \Delta_{n}(\{ p \, | \, \Delta_{n}'(p)=0 \}),
$$
that has zero Lebesgue measure by Sard's theorem.

It's clear that is enough to show the result for a single
$\psi\in  G$. So $\mu \pitchfork S_{n}$, that is, $S_{n}$ is shy.
\end{proof}

There are several nice papers  in the last years about properties of one-parametric families of circle diffeomorphisms. Several results of this kind can be restated in  our prevalence point of view. For example,  M. Tsujii \cite{TS2} has considered one-parametric families, $\{f_t = f + t\, | \, t \in S^1\}$ of orientation preserving circle diffeomorphisms, and he defines the following property: $f$ satisfies $(*)_{\beta}$, if
$$\text{ there exist finitely many }p,q \in \mathbb{Z},\text{ such that }|\rho(f) -\frac{p}{q}| < 1/ q^{2+\beta}.$$

J-C. Yoccoz \cite{Y} has proved that this condition implies that $f$ is $C^{r-2}$ conjugated to a rigid rotation $R_{\rho(f)}$, provided $f$ is $C^{r}$, for $r \geq 3$, and, $0< \beta< 1$.

M. Tsujii proved that the set
$$S_{\beta}=\{ t \in S^1 \, | \, \rho(f_t) \not\in \mathbb{Q}\text{ and } f_t \text{ not satisfy } (*)_{\beta}\},$$
has Lebesgue measure zero, since $df$ is of bounded variation.

 Using this property we can prove the next result. We note that  this application is, in a certain way, similar to the spirit of the Remark 5 in \cite{TS2}, but, not exactly the same (in our setting).
\begin{theorem}\label{Tsujii Version} A $C^{r}$ order preserving diffeomorphism of $\mathbf{S}^{1}$ is, generically, conjugated to a rigid rotation $R_{\rho}$, or, it has a rational rotation number, provided that $r \geq 3$.
\end{theorem}
\begin{proof}
We just have to choose a particular one-dimensional probe of type I,
$$\alpha_{\lambda}(x)=\psi_{\lambda,H}F=(1-\lambda)F + \lambda H= F + \lambda k , \; \lambda \in \mathbb{R},$$
$0<k<1$ and, as before, we denote  by $\mu$ the Lebesgue measure supported on the trace of $\alpha_{\lambda}$, where $\mu(A)$ means $\mu(\{\lambda \, | \, \alpha_{\lambda} \in A\}$). Since the Lebesgue measure is sigma finite we get from Tsujii's theorem, that $\mu(S_{\beta})=0$, where
$$S_{\beta}=\{ F \in \mathcal{H}_{0} \, | \, \rho(\pi(F)) \not\in \mathbb{Q}\text{ and } \pi(F) \text{ not satisfy } (*)_{\beta}\}.$$

We claim that $\mu(\psi(S_{\beta}) \cap \alpha_{\lambda})=0$,  or, equivalent, $\mu(S_{\beta} \cap \psi(\alpha_{\lambda}))=0$, for all $\psi \in G_{0}$. But this is a simple consequence of the fact that the dilatation map is absolutely continuous, with respect to the Lebesgue measure. Remember that
\begin{align*}
\psi(\alpha_{\lambda})(x,\lambda)&= (1-\delta) (F(x) + \lambda k) + \sum_{i=1}^{n}a_{i}G_{i}(x)\\
&= \psi(F) +  \lambda(1-\delta)k.
\end{align*}
We assume that $\psi(F)=(1-\delta) F(x) + \sum_{i=1}^{n}a_{i}G_{i}(x) \in \mathcal{H}_{0}$ (otherwise the path has zero measure) is of bounded variation too. So we can use Tsujii's theorem again to get $\mu(\psi(S_{\beta}) \cap \alpha_{\lambda})=0$.

Thus, $\mu \pitchfork S_{\beta}$, what means that $S_{\beta}$ is shy.
\end{proof}

Of course, the projection of the dense set $S_{\beta}^{c}$ is dense in $Diff^{r}(S^{1})$ (see Axiom~1). In \cite{TS2}, Tsujii obtains this result as a corollary of his main theorem.

\section{Hyperbolicity estimates}
In this section we will analyze a sightly different problem, that is, to make a quantitative analysis of the  hyperbolicity of a $C^{r}$, diffeomorphism, $r\geq 2$ (we need $r\geq 2$ in order to get an $C^{1}$ evaluation map). In Theorem~\ref{KS OrientPreservDiff} we prove that the orientation preserving diffeomorphisms are prevalently K-S, through the use of an special parametric family. If we use the quantitative measure of hyperbolicity, as in \cite{OY} Definition 7.6,
$$E_{\gamma}^{n}(F)=\{ x \in S^{1}\, | \, x \in Fix(F^{n}), \|dF^{n}(x) - 1\| \leq \gamma\},$$
for $\gamma >0$, then we can formulate the following problem:

Fix $F \in \mathcal{H}_{0}$ and choose $H \in A^{F} \subset \mathcal{H}_{0}$, given by
$A^{F}=\{H \in \mathcal{H}_{0} \, | \, H-F \neq 0\}.$ Consider the path of liftings of circle maps of type II, $\alpha_{\lambda}(F)(x)=(\psi_{\lambda,H}F) (x)=F(x) + \lambda (H-F)(x), \; \lambda \in \mathbb{R}.$\\

\emph{We want to estimate the Lebesgue measure of $$Z_{\gamma}^{n}(F)=\{ \lambda(x) \, | \, x \in E_{\gamma}^{n}(\alpha_{\lambda}(F))\text{ and } \alpha_{\lambda}(F) \in \mathcal{H}_{0}\}$$ where $\lambda(x)=\Delta_{n}(x)=\{\lambda \, | \,\alpha_{\lambda}(F)^{n}(x)=x \}$, is the  multi-valuated map introduced in the Section\ref{evaluationSection}.}
\begin{center}
\includegraphics[scale=0.30,angle=0]{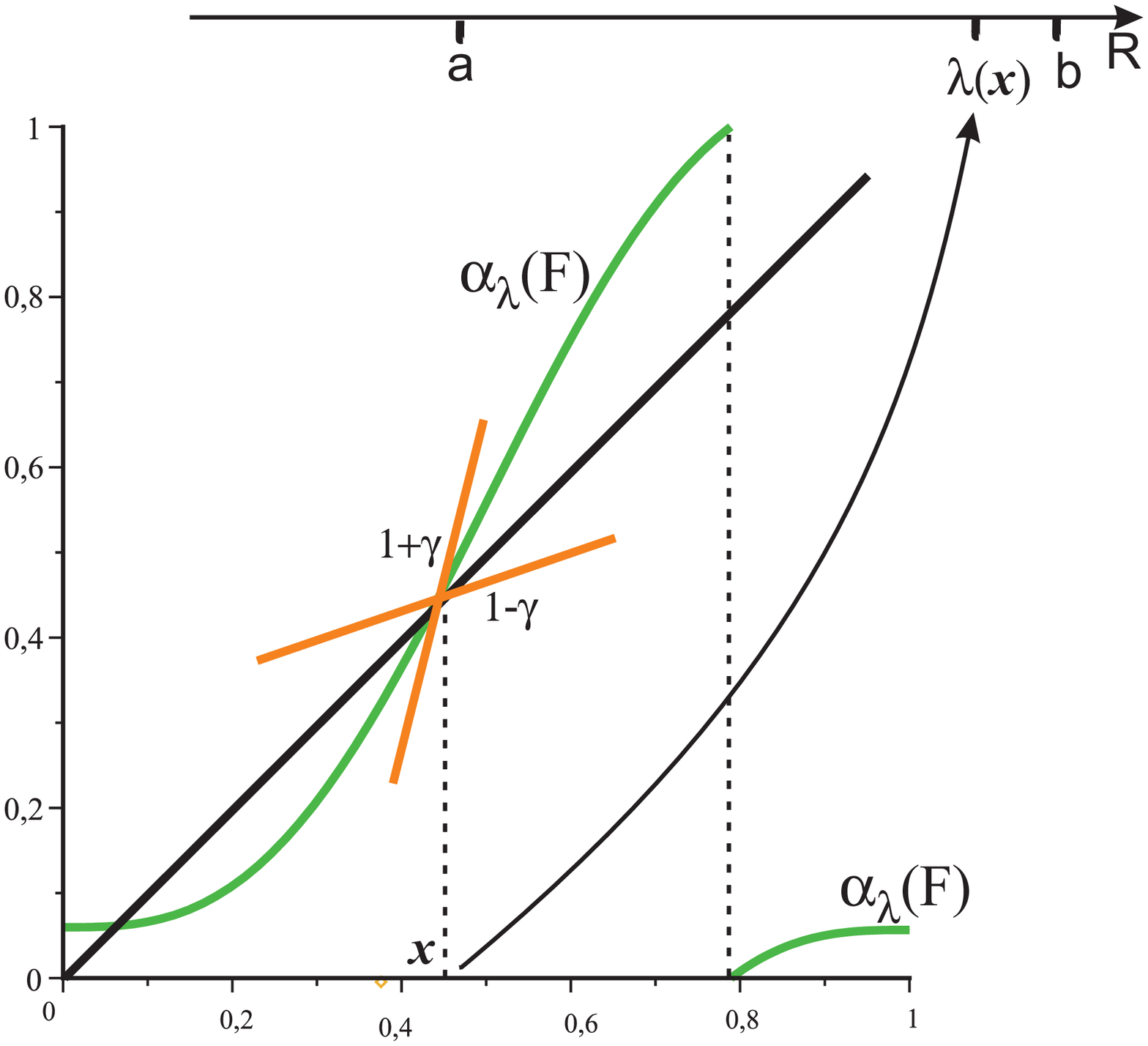}\\
\small{The set $Z_{\gamma}^{n}(F)$.}\\
\end{center}
We observe that the problem is well posed because the map $\Delta_{n}(x)$ is 1-periodic (otherwise such measure could be always $\infty$ or always 0), since $F$ and $H$ are liftings of circle diffeomorphisms. So this claim motivate the quantitative prevalent K-S theorem \textbf{(Q-K-S)}:

\begin{theorem}\label{QKSTheorem} Fix $F \in \mathcal{H}_{0}$ and choose $H \in A^{F}$, where $$A^{F}= \{ H \in \mathcal{H}_{0} \, | \, \inf_{x \in S^{1}} |H - F| =\sigma > 0\}.$$ Consider the path of circle diffeomorphisms
$\alpha_{\lambda}(F)(x),$
as above, and $\lambda(x)=\Delta_{n}(x)=\{\lambda \, | \,\alpha_{\lambda}(F)^{n}(x)=x \}$ the associated evaluation map. If $m$ is the Lebesgue measure in $\mathbb{R}$ then
$$m(Z_{\gamma}^{n}(F)) \leq c_{n} \frac{\gamma}{\sigma},$$
for some positive constant $c_{n}$ the depends of $\displaystyle \min_{ x \in \bigcup E_{\gamma}^{n}(\alpha_{\lambda}(F)) } \alpha_{\lambda}(F)'(x)$.
\end{theorem}
\begin{proof} Let $D_{n}$  be the domain of each evaluation and $(a,b) \supset [0,1]$ be the maximal interval such that $\alpha_{\lambda}(F) \in \mathcal{H}_{0}$. Then, $$Z_{\gamma}^{n}(F)=   \Delta_{n} \left( \bigcup_{ \lambda \in (a,b)} E_{\gamma}^{n}(\alpha_{\lambda}(F))\right), $$ so
$\Delta_{n}$ is a differentiable  $C^{1}$ function on this set, because $E_{\gamma}^{n}(\alpha_{\lambda}(F)) \subset D_{n}$ and $\lambda \in (a,b)$.

From Section \ref{evaluationSection} we get the formula,
$$\Delta_{n}'(x)=\frac{1- \frac{d\alpha_{\lambda}(F)^{n}}{dx}}{\frac{d\alpha_{\lambda}(F)^{n}}{dx} \sum_{k=0}^{n-1} (\frac{d\alpha_{\lambda}(F)^{k+1}}{dx})^{-1}(H-F)(\alpha_{\lambda}(F)^{k})},
$$
where $\lambda=\Delta_{n}(x)$.

From Schwartz \cite{SCH}, Chap. III, Thm. 3.1, we know that, if $D \subset \mathbb{R}^{n}$ is an open, and $g: D \to \mathbb{R}^{n}$ is $C^1$, and $E \subseteq D$ is a mensurable subset, then, $g(E)$ is measurable, and
$\displaystyle m(g(E))\leq\int_{E} |J_{g}| dm,$
where $J_{g}$ is the Jacobian determinant of $g$.

So, using  $\frac{d\alpha_{\lambda}(F)^{n}}{dx} \cdot \big( \frac{d\alpha_{\lambda}(F)^{i+1}}{dx} \big)^{-1}= \frac{d\alpha_{\lambda}(F)^{n-i}}{dx}  \circ \alpha_{\lambda}(F)^{i+1}$, and applying this property to our case we get,
\begin{align*}
m(\{ \lambda(x) \, | \, x \in E_{\gamma}^{n}(\alpha_{\lambda}(F))\text{ and } \alpha_{\lambda}(F) \in \mathcal{H}_{0}\}) = & \\
m(\Delta_{n}(\bigcup_{ \lambda \in (a,b)} E_{\gamma}^{n}(\alpha_{\lambda}(F)))) & \leq  \\
\int_{\bigcup E_{\gamma}^{n}(\alpha_{\lambda}(F))} | \Delta_{n}'(x)| dm & =  \\
\int_{\bigcup E_{\gamma}^{n}(\alpha_{\lambda}(F))}  \frac{\left\|1- d_{x} \alpha_{\lambda}(F)^{n}(x)\right\|}{ \sum_{i=0}^{n-1}\frac{d\alpha_{\lambda}(F)^{n-i}}{dx}  \circ \alpha_{\lambda}(F)^{i+1} \cdot \left\|(H-F)(\alpha_{\lambda}(F)^{i})\right\|} dm & \leq  \\
\frac{\gamma}{\sigma } \int_{\bigcup E_{\gamma}^{n}(\alpha_{\lambda}(F))}  \frac{1}{ \sum_{i=0}^{n-1}\frac{d\alpha_{\lambda}(F)^{n-i}}{dx}  \circ \alpha_{\lambda}(F)^{i+1} } dm & =\frac{\gamma}{\sigma } I_{n}
\end{align*}
where $I_{n}= \int_{\bigcup E_{\gamma}^{n}(\alpha_{\lambda}(F))}  \frac{1}{\sum_{i=0}^{n-1}\frac{d\alpha_{\lambda}(F)^{n-i}}{dx}  \circ \alpha_{\lambda}(F)^{i+1}}  dm $.

Let,  $\displaystyle u=\min_{ x \in \bigcup E_{\gamma}^{n}(\alpha_{\lambda}(F)) } \frac{d}{dx}\alpha_{\lambda}(F) (x)$ then $$\sum_{i=0}^{n-1}\frac{d\alpha_{\lambda}(F)^{n-i}}{dx}  \circ \alpha_{\lambda}(F)^{i+1} \geq \sum_{i=0}^{n-1} u^{n-i+1}=b_{n}.$$
Thus,
$\frac{1}{\sum_{i=0}^{n-1} u^{n-i+1}}=\frac{1}{b_{n}} = c_{n} < \infty$.
\end{proof}

Finally, we would like to point out that if $u>1$ on the above theorem, then $\sum \frac{1}{b_{n}}$ is obviously finite, on the other hand, we always have $1-\gamma \leq u \leq 1+ \gamma$, thus $\sum \frac{1}{b_{n}}$ may diverge.

\vspace{0.9cm}

\end{document}